\documentclass[10pt]{article}

\usepackage{amsmath,amssymb,amsthm,amscd}
\usepackage{graphicx}
\usepackage[english]{babel}
\usepackage{amsfonts}
\usepackage{epstopdf}
\usepackage{lmodern}
\usepackage{multirow}
\usepackage{showkeys}
\usepackage{hyperref}

\numberwithin{equation}{section}

\newcommand{\Z}{\mathbb{Z}}
\newcommand{\Q}{\mathbb{Q}}

\newcommand{\Fq}{\mathbb{F}_q}
\newcommand{\Zp}{\Z_p}
\newcommand{\FF}[1]{\mathbb{F}_{#1}}
\renewcommand{\hom}{\operatorname{Hom}}

\newcommand{\ch}{\operatorname{char}}
\newcommand{\rk}{\operatorname{rk}}
\newcommand{\floor}[1]{\left\lfloor #1 \right\rfloor}

\newtheorem{theorem}{Theorem}[section]
\newtheorem{lemma}[theorem]{Lemma}

\newtheorem*{conj}{Conjecture}

\makeatletter
\let\oldtheequation\theequation
\renewcommand\tagform@[1]{\maketag@@@{\ignorespaces#1\unskip\@@italiccorr}} 
\renewcommand\theequation{(\oldtheequation)}
\makeatother
\def\equationautorefname~{}

\title{$p$-Divisibility of the number of linear representations of an Abelian $p$-group}
\author{Chen Wang}

\begin{document}

\maketitle

\begin{abstract}
We establish lower bounds for the $p$-divisibility of the quantity\break $\#\hom(G,GL_n(\Fq))$, the number of homomorphisms from $G$ to a general linear group, where $G$ is an Abelian $p$-group. This is in analogy to the result of Krattenthaler and M\"{u}ller \cite{MR3383810} on homomorphisms to symmetric groups.
\end{abstract}

\section{Introduction}
The famous Frobenius Theorem \cite{zbMATH02656074} states the following.
\begin{theorem}[\cite{zbMATH02656074}]\label{thFrobenius}
If $n$ divides the size of a finite group $H$, then the number of elements in $H$ with order dividing $n$ is a multiple of $n$.
\end{theorem}
This result can be rephrased as saying that the number of homomorphisms from a cyclic group $C_n$ of order~$n$ to $H$ is divisible by $n$.

More generally, let $G$ and $H$ be two finite groups, and let $\#\hom(G,H)$ be the number of homomorphisms from $G$ to $H$. Yoshida \cite{MR1213788} proved that this quantity is divisible by $\gcd(|G|,|H|)$ if $G$ is Abelian. However, one can observe that in certain cases $\#\hom(G,H)$ is divisible by a much larger factor.

For an integer $a$ and a prime $p$, the \emph{p-adic valuation} $v_p(a)$ of~$a$ is defined as the maximal integer~$e$ such that $p^e$ divides~$a$. The function $v_p(\cdot)$ extends naturally to rational numbers. We shall sometimes refer to $v_p(a)$ as the ``$p$-divisibility of~$a$". The aim of this paper is to establish lower bounds on $v_p(\#\hom(G,H))$ for certain groups $G$ and $H$.

Throughout the paper, $C_n$ stands for the cyclic group of order $n$.

The problem of determining the $p$-divisibility of $\#\hom(G,H)$ for
Abelian groups $G$ dates back to 1951 when Chowla, Herstein and Moore \cite{MR0041849} proved that
\[
v_2(\#\hom(C_2,S_n))\geq \floor{\frac{n}{2}}-\floor{\frac{n}{4}}.
\]
Further results have been established \cite{MR1213788,MR1728396,MR1284066,MR1904386,MR1673480}. The most
general result in this direction was given by Krattenthaler and
M\"{u}ller (see \cite{MR3492164} for a different proof of a slightly
weaker result), of which we present a (weaker) asymptotic version below.
\begin{theorem}[{\cite[Theorem~25]{MR3383810}}]
Let $G$ be a finite Abelian $p$-group. There exists a constant $C=C(G)$ such that
\begin{equation}\label{eqKratt}
v_p(\#\hom(G,S_n))\geq Cn+O(1).
\end{equation}
\end{theorem}

A natural problem arises when one replaces the symmetric groups in the above results by other families of finite groups. In this paper, we will treat the case where $G$ is a finite Abelian $p$-group and $H\simeq GL_n(\Fq)$ is a general linear group over a finite field. This is equivalent to counting all linear representations of $G$ over $\Fq$. Numerical computations suggest that $v_p(\#\hom(G,GL_n(\Fq)))$ is roughly $O(n)$ if $\ch\Fq\neq p$, and $O(n^2)$ if $\ch\Fq=p$. We will call the former the \emph{non-modular case}, and the latter the \emph{modular case}.

 The main goal of this paper is to rigorously establish the above empirical observations: in the non-modular case, we succeed in providing a proof in full generality, see Theorems~\ref{thDivHomMain1} and \ref{thDivHomMain2}. On the other hand, in the modular case, we provide a proof in the case where $G$ is a cyclic group, see Theorem~\ref{thDivModularMain},
while we have to leave the general case as an open problem.

Section 2 of this paper deals with the non-modular case. We prove a linear lower bound on $v_p(\#\hom(G,GL_n(\Fq)))$ using a generating function result due to Chigira, Takegahara and Yoshida \cite{MR1783923}; see Theorems~\ref{thDivHomMain1} and \ref{thDivHomMain2}. In Section~3 treats the modular case. As mentioned above, we restrict ourselves to the case where $G$ is cyclic. We show that, when the order of $G$ is a power of~$p$, then $v_p(\#\hom(G,GL_n(\FF{p^v})))$ grows quadratically in~$n$. Finally, in Section~4, we present some remarks and a conjecture about a general lower bound on $v_p(\#\hom(G,H))$ for all finite groups $G$ and $H$.



\section{The non-modular case}
Let $p$ be a prime  number and $q$ a prime power.
Throughout this section, $\Fq$ denotes a finite field with $q$ elements, where $\ch\Fq\neq p$.

In this section, we will give a lower bound on $v_p(\#\hom(G,GL_n(\Fq)))$ by utilizing a product formula for the generating function for the homomorphism numbers given by Chigira, Takegahara and Yoshida (Theorem \ref{thGenFun}). We first present a general argument (Theorem \ref{thDivFMain}) for bounding the $p$-divisibility of coefficients of a power series in a product form. Attempting to insert the generating function directly into this argument leads to a lower bound that is frequently not tight (Theorem \ref{thDivHomLowerBound1}). However, by utilizing an alternate factorization of the generating function, one can improve this bound to be asymptotically tight in most cases (Theorems~\ref{thDivHomMain1} and \ref{thDivHomMain2}). The remaining cases will be treated in Section~\ref{ssSpecialCase}.

\subsection{The generating function}\label{ssGenFun}
In this subsection, we will begin by quoting the aforementioned product formula for the generating function for the homorphism numbers in the lemma below. This formula involves the dimensions of certain irreducible representations of the group $G$ over $\Fq$.
Lemmas~\ref{leNumberOfIrrRep}--\ref{thFormulaOfLogF} serve to compute the number of irreducible representations of a given dimension, and thereby make the generating function formula completely explicit.

\begin{lemma}[Yoshida et al.,\cite{MR1783923}]\label{thGenFun}
Let $q$ be a prime power, and $G$ be an Abelian group with $\gcd(|G|,q)=1$. Then we have
\begin{equation} \label{eqGenFunDef}
F(G,q;z)=\sum_{n=0}^{\infty}\frac{(-1)^n\#\operatorname{Hom}(G,GL_n(\Fq))}{q^{\binom{n}{2}}(q;q)_n}z^n=\prod_i\left(\sum_{n\geq0}\frac{(-1)^nz^{d_in}}{q^{d_i\binom{n}{2}}(q^{d_i};q^{d_i})_n}\right),
\end{equation}
where $d_1,d_2,\dots$ are the dimensions of the irreducible representations of $G$ over $\Fq$. 
\end{lemma}

Since the generating function $F(G,q;z)$ is represented as a product, we will consider its logarithm. Let
$$f(q,z)=\sum_{n\geq0}\frac{(-1)^nz^n}{q^{\binom{n}{2}}(q;q)_n},$$
and write
\begin{equation}\label{eqDefOfH}
\log f(q,z)=h(q,z).
\end{equation}
Theorem \ref{thGenFun} states that $$F(G,q;z)=\exp\left(\sum_{i}h(q^{d_i},z^{d_i})\right).$$
To give an explicit form for $F(G,q;z)$, we explicitly enumerate the irreducible representations in the lemma below.

\begin{lemma}\label{leNumberOfIrrRep}
Let $G$ be an Abelian group, and $q$ be a prime power such that $\gcd(|G|,q)=1$. Then, for any $d\in\mathbb{Z}^+$, the number of irreducible $d$-dimensional representations of $G$ over $\Fq$ is equal to the number of length-$d$ cycles of the map $g\mapsto g^q$ in the group $G$.
\end{lemma}

\begin{proof}
Since $G$ is Abelian, every irreducible representation decomposes into a direct sum of 1-dimensional representations (or linear characters) over a sufficiently large extension of $\Fq$, namely $\FF{q^r}$ for some $r$ such that $\exp G\mid (q^r-1)$. Here $\exp G$ is the least common multiple of the order of elements in $G$. Let $\hat{G}=\{\chi_1,\chi_2,\dots,\chi_{|G|}\}$ be the \emph{character group} of $G$ consisting of all linear characters of $G$.

Any representation of $G$ over $\FF{q^r}$ is an integral linear combination of these linear characters. Among those, the representations of $G$ over $\Fq$ can be characterized by the invariance under the action of the Frobenius map $\chi\mapsto\chi^q$. Consequently, any irreducible representation over $\Fq$ is a sum over characters within a single orbit in $\hat{G}$ under this action. Therefore, the set of irreducible representations is in bijection with the set of these orbits, and the dimension of such a representation is equal to the size of its corresponding orbit. The lemma follows by noticing that $\hat{G}\simeq G$.
\end{proof}

To count the cycles in Lemma~\ref{leNumberOfIrrRep}, we also need a fundamental divisibility result concerning the $p$-divisibility of $q^n-1$.
\begin{lemma}\label{leDivQNMinus1}
Suppose that $p$ is a prime, $p\nmid q$, and the order of $q\pmod{p}$ is $d$. Suppose that $\lambda_i=v_p(q^{dp^i}-1).$ Then, for every integer $n\geq0$, we have:
\begin{itemize}
\item If $d\nmid n$, then $v_p(q^n-1)=0$.
\item If $d\mid n$, then $v_p(q^n-1)=\lambda_{v_p(n)}$.
\end{itemize}
It follows that $v_p((q;q)_n)=\sum_{i\geq 0}\lambda_{i}(\floor{n/dp^i}-\floor{n/dp^{i+1}})$.
\end{lemma}
\begin{proof}
The first part is obvious. For the second part, we write $n=ldp^i$ where $p\nmid l$. Then we have
\begin{equation}
\frac{q^{n}-1}{q^{dp^i}-1}= 1+q^{dp^i}+q^{2dp^i}+\dots+q^{(l-1)dp^i}\equiv l\not\equiv 0\pmod{p},
\end{equation}
so $v_p(q^n-1)=v_p(q^{dp^i-1})=\lambda_i$.
\end{proof}

Based on the two lemmas above, we are ready to give an explicit expression for the generating function $F(G,q;z)$ given in Theorem~\ref{thGenFun}.

\begin{lemma}\label{thFormulaOfLogF}
Let $G=\prod_{j=1}^{r}C_{p^j}^{k_j}$ be a finite Abelian $p$-group, where $k_1,k_2,\dots,k_r$ are non-negative integers and $k_r>0$. Let $q$ be a prime power such that $p\nmid q$, and $d$ be the order of $q$ modulo $p$. Then there exists an integer sequence $c_0,c_1,\dots$ such that
\begin{multline}\label{eqFormulaOfLogF}
\log F(G,q;z)\\=\left(h_1(q,z)-\frac{1}{d}h_d(q,z)\right)+\frac{1}{d}\sum_{i=0}^{\infty}p^{c_{\lambda_i}-i}\left(h_{dp^i}(q,z)-\frac{1}{p}h_{dp^{i+1}}(q,z)\right),
\end{multline}
where the sequences $c_i$ and $\lambda_i$ are defined by
\begin{equation}
c_i=\sum_{j=1}^{r}\min(i,j)k_j
\end{equation}
and
\begin{equation}
\lambda_i=v_p(q^{dp^i}-1).
\end{equation}
\end{lemma}
\begin{proof}
Theorem \ref{thGenFun} and Lemma~\ref{leNumberOfIrrRep} imply the representation
\begin{equation}\label{eqLogF1}
\log F(G,q;z)=\sum_{e\geq 1}n_e h_e(q,z),
\end{equation}
where $n_e$ is the number of length-$e$ cycles of the map $g\mapsto g^q$ in the group $G$.

An element $g\in G$ belongs to a cycle with length dividing $e$ if and only if the order of $g$ divides $q^e-1$. Since $G$ is a $p$-group, the order of $g$ must be some power of $p$. Lemma~\ref{leDivQNMinus1} implies that the only possible cycle lengths are $1,d,dp,dp^2,\dots$. Therefore, for every $i\geq0$, the elements with corresponding cycle length dividing $dp^i$ are exactly those with order dividing $p^{\lambda_i}$.

We note that the number of elements in $G$ with order dividing $p^i$ is given by $p^{c_i}$ (since they form a subgroup isomorphic to $\prod_{j=1}^{r}C_{p^{\min\{i,j\}}}^{k_j}$). Therefore, we have
\begin{equation}
\sum_{e\mid dp^i}en_e=p^{c_{\lambda_i}}.
\end{equation}
Using this equation we can calculate the quantity $n_e$ for all $e$.
If $d>1$, then we have
\begin{align*}
n_1 &= 1 \\
n_d &= \frac{1}{d}(p^{c_{\lambda_0}}-1) \\
n_{dp^i} &= \frac{1}{dp^i}(p^{c_{\lambda_i}}-p^{c_{\lambda_{i-1}}}).
\end{align*}
and if $d=1$, then we have
\begin{align*}
n_1 &= p^{c_{\lambda_0}} \\
n_{p^i} &= \frac{1}{p^i}(p^{c_{\lambda_i}}-p^{c_{\lambda_{i-1}}}).
\end{align*}
Insertion of these results into \ref{eqLogF1} finishes the proof of the lemma.
\end{proof}

\subsection{Bounding the $p$-divisibility of a power series}\label{ssFramework}
In this section, we present a general framework for bounding the quantity $v_p([z^n]F(G,q;z))$ from below. We begin with a lemma that relates the $p$-divisibility of the coefficients of a power series and the coefficients of its exponential.

\begin{lemma}\label{leExpPDiv}
Let $f(z)=\sum_{n\geq 1}a_nz^n$ be a power series, and $b\in\mathbb{Z}$. Suppose that $v_p(a_n)\geq b$ for all $n$, and $v_p(a_1)=b$. Then we can say the following about the quantity $v_p([z^n]e^{f(z)})$:
\begin{itemize}
\item If $b\leq0$, then $v_p([z^n]e^{f(z)})=nb-v_p(n!)$ for all $n\geq 0$.
\item If $b>0$, then $v_p([z^n]e^{f(z)})\geq b$ for all $n>0$.
\end{itemize}
\end{lemma}
\begin{proof}
The $b>0$ case is obvious. If $b\leq 0$, we have
\begin{align*}
[z^n]\exp f(z)&=\sum_{m=1}^{\infty}\frac{1}{m!}[z^n]f^m(z)\\
&=\sum_{m=1}^{\infty}\frac{1}{m!}\sum_{n_1+n_2+\dots+n_m=n}a_{n_1}a_{n_2}\dots a_{n_m}.
\end{align*}
Using the relationship $v_p\big(\sum x_i\big)\geq \min v_p(x_i)$, we have
\begin{align*}
v_p([z^n]\exp f(z))&\geq\min_{m\ge1}\min_{n_1+n_2+\dots+n_m=n}(mb-v_p(m!))\\
&=\min_{1\le m\le n}(mb-v_p(m!))\\
&=nb-v_p(n!).
\end{align*}
Equality holds because $nb-v_p(n!)$ is strictly decreasing, so the unique minimizer is at $m=n$ and $n_1=n_2=\dots=n_m=1$.
\end{proof}

Since we have asymptotically $v_p(n!)=\frac{n}{p-1}+O(1)$, the above lemma states that for $b\leq0$, the $p$-divisibility of $[z^n]\exp f(z)$ is roughly given by $n(b-\frac{1}{p-1})$. This gives an intuition for the theorem below, which will be the key for proving the main results of this section, given in Theorems~\ref{thDivHomMain1} and \ref{thDivHomMain2} below.

\begin{theorem}\label{thDivFMain}
Let $p$ be a prime. Suppose that $$F(z)=F^\dag(z)\exp\left(\sum_{i=0}^{\infty}f_i\left(z^{dp^i}\right)\right),$$ where $f_i\in p^{b_i}\Zp[[z]]$ for some $b_i\in\Z$, and $F^\dag(z)\in\Zp[[z]]$. Let
$$l={\arg\min}_i p^{-i}\left(b_i-\frac{1}{p-1}\right).$$
Then
\begin{equation}
v_p([z^n]F(z)^{\alpha})\geq b_l\floor{\frac{n}{dp^l}}-v_p\left(\floor{\frac{n}{dp^l}}!\right)
\end{equation}
for any $\alpha\in\Zp$ such that $v_p(\alpha)=0$. Here, $\arg\min_ih_i$ means the smallest index $i$ such that $h_i\le h_j$  for all $j$. If $l$ is the unique minimizer, then this bound is tight whenever $dp^l\mid n$.
\end{theorem}

\begin{proof}
Lemma \ref{leExpPDiv} gives:
\begin{itemize}
\item If $b_i\leq0$, $v_p([z^{p^in}]\exp f_i(z))=n b_i-v_p(n!)$ for every $n\geq0$.
\item If $b_i>0$, $v_p([z^{p^in}]\exp f_i(z))\geq 0$ for every $n>0$.
\end{itemize}

Therefore, we may write
\begin{equation*}
[z^n]F(z)=\sum_{n_0+n_1p+\dots=n}\prod_{i\geq0}\left([z^{n_ip^i}]F_i(z)\right),
\end{equation*}
and the inequality $v_p(\sum x_i)\geq \min v_p(x_i)$ gives
\begin{align}\notag
v_p([z^n]F(z))&\geq\min_{n_0+n_1p+\dots=n}\sum_{i\geq0}\underbrace{v_p\left([z^{n_ip^i}]F_i(z)\right)}_{\text{$>0$ when both $b_i>0$ and $n_i>0$}}\\
\notag
&\geq \min_{n_0+n_1p+\dots=n}\sum_{b_i\leq 0}v_p\left([z^{n_ip^i}]F_i(z)\right)\\
&=\min_{n_0+n_1p+\dots=n}\sum_{b_i\leq 0}\left(b_in_i-v_p\left(n_i!\right)\right)\label{eqTemp1}
\end{align}

We claim that, for any $i$, we have
\begin{equation}\label{eqTemp2}
b_in_i-v_p\left(n_i!\right)\geq b_lp^{i-l}n_i-v_p\left((p^{i-l}n_i)!\right),
\end{equation}
and the inequality is strict whenever $n_i>0$ and $p^{-i}\left(b_i-\frac{1}{p-1}\right)> p^{-l}\left(b_l-\frac{1}{p-1}\right)$.
In order to see this, we proceed as follows:
\begin{align*}
&\left(b_in_i-v_p\left(n_i!\right)\right)- \left(b_lp^{i-l}n_i-v_p\left((p^{i-l}n_i)!\right)\right)\\
&=n_i(b_i-b_lp^{i-l})+v_p\left((p^{i-l}n_i)!\right)-v_p\left(n_i!\right)\\
&=n_i(b_i-b_lp^{i-l}+\frac{p^{i-l}-1}{p-1})\\
&=n_ip^i\left(p^{-i}\left(b_i-\frac{1}{p-1}\right)-p^{-l}\left(b_l-\frac{1}{p-1}\right)\right)\\
&\geq0.
\end{align*}
Combination of \ref{eqTemp1} and \ref{eqTemp2} leads to
\begin{align*}
v_p([z^n]F(G,q;z))&\geq\min_{n_0+n_1p+\dots=n}\sum_{b_i\leq 0}\left(b_in_i-v_p\left(n_i!\right)\right)\\
&\geq \min_{n_0+n_1p+\dots=n}\sum_{b_i\leq 0}\left(b_lp^{i-l}n_i-v_p\left((p^{i-l}n_i)!\right)\right)\\
&=\min_{n_0+n_1p+\dots=n}\left(b_l\sum_{b_i\leq 0}p^{i-l}n_i-\sum_{b_i\leq 0}\left(v_p\left((p^{i-l}n_i)!\right)\right)\right).
\end{align*}
We notice that $\sum_{b_i\leq 0}p^{i-l}n_i\leq \floor{np^{-l}}$, as well as the fact that
\begin{align*}
\sum_{b_i\leq 0}v_p\left((p^{i-l}n_i)!\right)&\leq v_p\left(\left(\sum_{b_i\leq 0}(p^{i-l}n_i)\right)!\right)\\
&\leq v_p\left(\floor{np^{-l}}!\right),
\end{align*}
to conclude
\begin{equation*}
v_p([z^n]F(G,q;z))\geq b_l\floor{\frac{n}{dp^l}}-v_p\left(\floor{\frac{n}{dp^l}}!\right).
\end{equation*}
\end{proof}

\subsection{Auxiliary results}\label{ssAux}
This part contains some divisibility results for the coefficients of certain linear combinations of various series $h(q^d,z^d)$ as defined in \ref{eqDefOfH}. They will be used when these series are substituted into Theorem~\ref{thDivFMain} in Subsection~\ref{ssMain}. We defer the proofs of the auxiliary results of this subsection to the Appendix.

\begin{theorem}\label{thLogOfGeneratingFunction}
The power series $h(q,z)$, as defined in \eqref{eqDefOfH}, has the form
\[
h(q,z)=\sum_{n\geq1}\frac{(-1)^{n-1}P_{n}(q)}{n q^{\binom{n}{2}}(q^n-1)}z^n,
\]
where $P_{n}\in\mathbb{Z}[x]$ is a polynomial such that $\deg P_{n}=\binom{n}{2}$, and $P_{dn}(\omega_d)=\binom{2n-1}{n-1}$, for any primitive $d$-th root of unity $\omega_d$.
\end{theorem}.

\begin{theorem}\label{thDivLogFCase1}
Let $q$ be a prime power with $q\equiv1 \pmod{p}$. Then we have
\begin{equation}
(q-1)\left(h(q,z)-\frac{1}{p}h(q^p,z^p)\right)\in \Zp[[z]].
\end{equation}
\end{theorem}

\begin{theorem}\label{thDivLogFCase2}
Let $q$ be an odd prime power. Then we have
\begin{equation}
(q-1)\left(h(q,z)-\frac{1}{2}h(q^2,z^2)\right)\in z-\frac{z^2}{q}+2z^3\Zp[[z]].
\end{equation}
\end{theorem}

\begin{theorem}\label{thDivLogFCaseD}
Let $q$ be a prime power with $q\not\equiv1 \pmod{p}$. Suppose that $d$ is the order of $q$ in $\mathbb{Z}/p\mathbb{Z}$. Then
\begin{equation}
\exp\left(h(q,z)-\frac{1}{d}h(q^d,z^d)\right)\in \Zp[[z]].
\end{equation}
\end{theorem}

\subsection{A first estimate}\label{ssMain}
We will now begin to actually bound the $p$-divisibility of the coefficients of the generating function $F(G,q;z)$ from Theorem~\ref{thGenFun} by using Theorems \ref{thFormulaOfLogF}, \ref{thDivFMain} and \ref{thDivLogFCase1}.

\begin{theorem}\label{thGenFunBound1}
Let $b_i=c_{\lambda_i}-\lambda_i-i$, where the sequences $c_i$ and $\lambda_i$ are as defined in Theorem \ref{thFormulaOfLogF}. Furthermore, let $l$ be the index that minimizes\break $p^{-l}(b_l-\frac{1}{p-1})$. Then we have the estimate
\begin{equation}\label{eqGenFunBound1}
v_p([z^n]F(G,q;z))\geq b_l\left\lfloor\frac{n}{dp^l}\right\rfloor-v_p\left(\left\lfloor\frac{n}{dp^l}\right\rfloor!\right).
\end{equation}
\end{theorem}
\begin{proof}
Based on formula \ref{eqFormulaOfLogF}, we define for every $i\geq0$
\begin{equation}
F_i(z)=\exp\left(d^{-1}p^{c_{\lambda_i}-i}\left(h(q^{dp^i},z^{dp^i})-\frac{1}{p}h(q^{dp^{i+1}},z^{dp^{i+1}})\right)\right)
\end{equation}
and
\begin{equation}
F_*(z)=\exp\left(h(q,z)-\frac{1}{d}h(q^{d},z^{d})\right),
\end{equation}
so that $F(G,q;z)=F_*(z)\prod_i F_i(z)$.

We estimate the coefficients of $\log F_i(z)$ by Theorem \ref{thDivLogFCase1}:
\begin{align*}
v_p([z^{dp^{i}n}]\log F_i(z))&\geq c_{\lambda_i}-i+v_p\left([z^{dp^{i}n}]\left(h(q^{dp^i},z^{dp^i})-\frac{1}{p}h(q^{dp^{i+1}},z^{dp^{i+1}})\right)\right)\\
&=c_{\lambda_i}-i+v_p\left([z^n]\left(h(q^{dp^{i}},z)-\frac{1}{p}h(q^{dp^{i+1}},z^p)\right)\right)\\
&\geq c_{\lambda_i}-i-v_p(q^{p^{i}}-1)\\
&=c_{\lambda_i}-\lambda_i-i\\
&=b_i.
\end{align*}
Therefore, $\log F_i(z)\in p^{b_i}\Zp[[z]]$. We substitute this into Theorem \ref{thDivFMain} to get \ref{eqGenFunBound1}.
\end{proof}

In order to transform this bound into a lower bound for $v_p(\#\hom(G,GL_n(\Fq)))$, we need a result about the sequence $\lambda_i$.

\begin{lemma}\label{lePropertiesOfLambda}
Let $p$ be a prime, and $q$ be a prime power not divisible by $p$. Let $\lambda_i=v_p(q^{dp^i}-1)$ be as defined in \autoref{thFormulaOfLogF}, where $d$ is the order of $q$ mosulo $p$. Then the following holds:
\begin{enumerate}
\item\label{caseLambda1} If either $p\geq 3$ or $\lambda_0\geq 2$, then $\lambda_i=\lambda_0+i$ for all $i\geq0$.
\item\label{caseLambda2} If $p=2$ and $\lambda_0=1$, then $\lambda_i=\lambda_1+i-1$ for all $i\geq1$.
\end{enumerate}
\end{lemma}
\begin{proof}
It suffices to prove that, except for the case where $p=2$ and $\lambda_i=1$, we always have $\lambda_{i+1}-\lambda_i=1$. Indeed, using the definition of $\lambda_i$, and writing $r=q^{dp^i}$ for simplicity, we conclude
\begin{align*}
\lambda_{i+1}-\lambda_i&=v_p\left(\frac{r^p-1}{r-1}\right)\\
&=v_p\left(1+r+\dots+r^{p-1}\right).
\end{align*}
Again, from the definition of $\lambda_i$, we have $p^{\lambda_i}\mid (r-1)$. Suppose that $r\equiv a p+1\pmod{p^2}$ for some $a$ with $0\leq a<p$. Then calculating modulo $p^2$, we have
\begin{equation}
\begin{split}
1+r+\dots+r^{p-1}&\equiv 1+(ap+1)+(2ap+1)+\dots+((p-1)ap+1)\\
&\equiv p+\frac{a(p+1)}{2}p^2 \pmod{p^2}.
\end{split}
\end{equation}
We claim that $\frac{a(p+1)}{2}\in\Z$ unless $p=2$ and $\lambda_i=1$. If $p\geq 3$ then this is obvious. If $p=2$ and $\lambda_i>1$ then we have $p^2\mid (r-1)$, so $a=0$ and the claim also holds. This claim yields $1+r+\dots+r^{p-1}\equiv p\pmod{p^2}$, and therefore $v_p\left(1+r+\dots+r^{p-1}\right)=1$, as desired.
\end{proof}

We are ready to establish a first lower bound for $v_p(\#\hom(G,GL_n(\Fq)))$.

\begin{theorem}\label{thDivHomLowerBound1}
Let $l$ and $b_l$ be as defined in Theorem \ref{thGenFunBound1}. Then we have
\begin{equation}
v_p(\#\hom(G,GL_n(\Fq)))\geq b_l\floor{\frac{n}{dp^l}}+\sum_{i=0}^l\floor{\frac{n}{dp^i}}+(\lambda_0-1)\floor{n/d}
\end{equation}
if either $p\geq 3$ or $\lambda_0\geq 2$, and
\begin{equation}
v_p(\#\hom(G,GL_n(\Fq)))\geq b_l\floor{\frac{n}{2^l}}+\sum_{i=0}^l\floor{\frac{n}{2^i}}+(\lambda_1-2)\floor{n/2}
\end{equation}
if $p=2$ and $\lambda_0=1$.
\end{theorem}
\begin{proof}
The definition \ref{eqGenFunDef} of the generating function $F(G,q;z)$ leads to
\begin{align*}
v_p(\#\hom(G,GL_n(\Fq)))&=v_p([z^n]F(G,q;z))+v_p(|GL_n(\Fq)|)\\
&\geq b_l\left\lfloor\frac{n}{dp^l}\right\rfloor-v_p\left(\left\lfloor\frac{n}{dp^l}\right\rfloor!\right)+\lambda_0\floor{n/d}\\
&\kern2cm
+\sum_{i\geq 0}(\lambda_{i+1}-\lambda_i) \floor{n/dp^i},
\end{align*}
where we used Theorem \ref{thGenFunBound1} to get the inequality.

We know that $v_p\left(\floor{\frac{n}{dp^l}}!\right)=\sum_{i>l}\floor{\frac{n}{dp^i}}$, and that $\lambda_{i+1}-\lambda_i=1$ whenever $i\geq1$. This allows us to eliminate the infinite summation and conclude that
\begin{multline}
v_p(\#\hom(G,GL_n(\Fq)))\\\geq b_l\floor{\frac{n}{dp^l}}+\sum_{i=0}^l\floor{\frac{n}{dp^i}}+(\lambda_0-1)\floor{n/d}+(\lambda_1-\lambda_0-1)\floor{n/dp}.
\end{multline}
According to Lemma~\ref{lePropertiesOfLambda}, if either $p\geq 3$ or $\lambda_0\geq 2$ the term $\lambda_1-\lambda_0-1$ vanishes, while otherwise the term $\lambda_0-1$ vanishes. This yields the two forms in the theorem.
\end{proof}

\subsection{Improving the bound}

Now, we have, in principle, an algorithm for computing a lower bound for computing $v_p(\#\operatorname{Hom}(G,GL_n(\Fq)))$ given $G$ and $q$. It turns out that the bound given in \autoref{thDivHomLowerBound1} is tight only in some cases. To identify the cases where it is not tight, to give improved bounds in these cases, as well as to give a way to calculate the bound from $p$, $q$ and $G$ without explicitly writing out the sequence $b_i=c_{\lambda_i}-\lambda_i-i$, we need to look into the properties of $b_i$. These come from the interaction between $c_i$ (determined by the group $G$) and $\lambda_i$ (determined by $p$ and $q$).

We illustrate the cases where the bound given in \autoref{thDivHomLowerBound1} is not\break tight by numerical calculations. In \autoref{tbComputation} we list some numerical results for\break $v_p([z^n]F(G,q;z))$, and compare them with the calculations based on \autoref{thDivHomLowerBound1}, for some selected example cases. The minimizer(s) of $p^{-i}(b_i-\frac{1}{p-1})$ are shown in \textbf{bold}. This table suggests that if the bound given by \autoref{thDivHomLowerBound1} is not tight, then either $b_l=0$ and $b_{l+1}=-1$, or $p=2$, $b_l=-1$ and $b_{l+1}=-3$.

\begin{table}
\begin{tabular}{|c|c|c|p{4.6cm}|c|l|}
  \hline
  $G$ & $p$ & $q$ & Actual lower bound for $v_p(\#\hom(G,GL_n(\Fq)))$ & $b_i$ & Theorem \ref{thDivHomLowerBound1} \\
  \hline
  $C_2$ & 2 & 3 & $n-\floor{n/2}$ & $0,\mathbf{-3},-5,-7,\dots$& $n-\floor{n/2}$\\
  $C_4^3$ & 2 & 47 & $n+4\floor{n/2}-\floor{n/4}$ & $2,0,\mathbf{-2},-4,\dots$ & $n+4\floor{n/2}-\floor{n/4}$ \\
  $C_{27}$ & 3 & 163 & $3n$ & $\mathbf{-1},-3,-5,-7,\dots$ & $3n$ \\
  $C_9$ & 3 & 17 & $2\floor{n/2}-\floor{n/6}$ & $0,\mathbf{-2},-4,-6,\dots$ & $2\floor{n/2}-\floor{n/6}$ \\
  \hline
  $C_2$ & 2 & 5 & $n+\floor{n/2}-\floor{n/4}$ & $\mathbf{-1,-3},-5,-7,\dots$ & $n$\\
  $C_4^3$ & 2 & 31 & $n+4\floor{n/2}+\floor{n/4}-\floor{n/8}$ & $2,\mathbf{-1,-3},-5,\dots$ & $n+4\floor{n/2}$ \\
  $C_{27}$ & 3 & 7 & $n+\floor{n/3}+\floor{n/9}-3\floor{n/27}$ & $\mathbf{0,-1},-2,-4,\dots$ & $n$ \\
  $C_9$ & 3 & 5 & $\floor{n/2}+\floor{n/6}-2\floor{n/18}$ & $\mathbf{0,-1},-3,-5,\dots$ & $\floor{n/2}$ \\
  \hline
\end{tabular}
\caption{Selected examples of using Theorem \ref{thDivHomLowerBound1} and actual numerical results.}\label{tbComputation}
\end{table}

We will find a way to circumvent the first case in the following arguments. To begin with, we consider an extreme example where Theorem \ref{thDivHomLowerBound1} does not give a tight bound by taking $G=C_{p^k}$ and $k\to\infty$. The following lemma describes the situation under such a limit.
\begin{lemma}
Let $F$ be defined as in \eqref{eqGenFunDef}. Then 
\begin{multline}\label{eqFormulaOfLogFInfty}
\lim_{k\to\infty}\log F(C_{p^k},q;z)\\=\left(h(q,z)-\frac{1}{d}h(q^d,z^d)\right)+\frac{1}{d}\sum_{i=0}^{\infty}p^{\lambda_i-i}\left(h(q^{dp^i},z^{dp^i})-\frac{1}{p}h(q^{dp^{i+1}},z^{dp^{i+1}})\right).
\end{multline}
\end{lemma}
\begin{proof}
In this case we have $c_i=i$ for all $0\leq i\leq k$, and $c_i=k$ for all $i\geq k$. The claim follows by using \autoref{thFormulaOfLogF} and letting $k\to+\infty$.
\end{proof}

We use Theorem \ref{thDivHomLowerBound1} to calculate a lower bound on $$v_p(\lim_{k\to\infty}[z^n]F(C_{p^k},q,z)).$$ Here we have $b_i=c_{\lambda_i}-\lambda_i-i=-i$, so the minimizer of $p^{-i}(b_i-\frac{1}{p-1})$ is achieved at $i=0$ and $i=1$. The corresponding lower bound is given by
\begin{equation}
v_p(\lim_{k\to\infty}[z^n]F(C_{p^k},q,z))\geq -v_p(\floor{\frac{n}{d}}!).
\end{equation}

However, the lower bound can actually be improved to $0$ by using Lemma~\ref{leDivOfFCaseInfty}. This discrepancy is actually the underlying reason for most cases of non-tightness in Theorem \ref{thDivHomLowerBound1}.

In order to improve the bound, we exploit Lemma~\ref{leDivOfFCaseInfty} by considering an alternative decomposition of $\log F(G,q,z)$:
\begin{multline}\label{eqFormulaOfLogFAlt}
\log F(G,q,z)=p^s\lim_{k\to\infty}\log F(C_{p^k},q;z)-(p^s-1)\left(h_1(q,z)-\frac{1}{d}h_d(q,z)\right)\\
+\frac{1}{d}\sum_{i=0}^{\infty}\left(p^{c_{\lambda_i}-i}-p^{\lambda_i+s-i}\right)\left(h_{dp^i}(q,z)-\frac{1}{p}h_{dp^{i+1}}(q,z)\right).
\end{multline}
Here, using Theorem \ref{thDivLogFCaseD} and Lemma~\ref{leDivOfFCaseInfty}, we conclude that the exponential of the first two terms in \ref{eqFormulaOfLogFAlt} both belong to $\Zp[[z]]$.

To illustrate the point in considering this form of $\log F(G,q,z)$, we write
$$\log F_i^*(z)=\frac{1}{d}\left(p^{c_{\lambda_i}-i}-p^{\lambda_i+s-i}\right)\left(h_{dp^i}(q,z)-\frac{1}{p}h_{dp^{i+1}}(q,z)\right).$$
Note that, using Lemma~\ref{lePropertiesOfC}, we see that $\log F_i^*(z)$ vanishes whenever $r'\leq\lambda_i\leq r$. We can prove that $b_{i+1}=b_i-1$ if and only if $r'\leq\lambda_i< r$, so all the cases where $b_l=0$ and $b_{l+1}=-1$ disappear in the new decomposition.

Based on this, we let $a_i=+\infty$ if $r'\leq \lambda_i\leq r$, and $a_i=b_i$ otherwise, so that $\log F_i^*(z)\in p^{a_i}\Zp[[z]]$.

To prove the feasibility of replacing $b_i$ by $a_i$, we first list some properties of $c_i$ and $b_i$ in Lemmas~\ref{lePropertiesOfC} and \ref{lePropertiesOfB}.

\begin{lemma}\label{lePropertiesOfC}
Let $G=\prod_{i=1}^{r}C_{p^i}^{k_i}$ be a finite Abelian $p$-group, and $$c_i=\sum_{j=1}^{r}\min(i,j)k_j,$$ as defined in \autoref{thFormulaOfLogF}. Suppose that $G=C_{p^r}\times G_0$, where $\exp G_0=p^{r'}$ and $|G_0|=p^{s}$. Then the following holds:
\begin{itemize}
\item $c_i$ is concave and non-decreasing with respect to $i$.
\item For $i\geq r$, $c_i=r+s$.
\item For $r'\leq i\leq r$, $c_i=i+s$.
\item For $i<r'$, $c_{i+1}-c_i\geq2$.
\end{itemize}
\end{lemma}

Having already equipped ourselves with Lemmas~\ref{lePropertiesOfLambda} and \ref{lePropertiesOfC}, we now combine these lemmas to describe some properties about $b_i$.

\begin{lemma}\label{lePropertiesOfB}
Under the assumptions of Lemmas \ref{lePropertiesOfLambda} and \ref{lePropertiesOfC}, we have $b_{i+1}-b_{i}\geq -2$ for all $i$, unless $i=0$, $p=2$, and $\lambda_0=1$. In the latter case, we have $b_i\geq0$. Moreover, if $\lambda_i>r$, then $b_{i+1}-b_{i}=-2$.
\end{lemma}

\begin{proof}
We have $b_i=c_{\lambda_i}-\lambda_i-i$. Furthermore, since the sequences $c_i$ and $\lambda_i$ are nondecreasing, we have
\begin{equation}
b_{i+1}-b_i=c_{\lambda_{i+1}}-c_{\lambda_i}-(\lambda_{i+1}-\lambda_i)-1\geq-(\lambda_{i+1}-\lambda_i)-1.
\end{equation}
\autoref{lePropertiesOfLambda} tells us that $\lambda_{i+1}-\lambda_i=1$ unless $p=2$, $\lambda_0=1$, and $i=0$, which implies our claim. If the latter condition should be violated, then we have $b_0=c_1-1\geq 0$.

If $\lambda_i>r\geq1$, \autoref{lePropertiesOfC} says that $c_{\lambda_i}=c_{\lambda_{i+1}}=r+s$, so $b_{i+1}-b_i=-(\lambda_{i+1}-\lambda_i)-1=-2.$
\end{proof}

Based on the properties in Lemma~\ref{lePropertiesOfB}, we can show that by using the alternative form \ref{eqFormulaOfLogFAlt} and replacing $b_i$ by $a_i$ in Theorem \ref{thGenFunBound1}, we can eliminate most of the cases where Theorem \ref{thDivHomLowerBound1} does not give tight bounds, and at the same time provide an easier way to calculate the minimizer in Theorem \ref{thDivHomLowerBound1}.

\begin{lemma}\label{leFormMinimizerOfAL}
Let $l$ be the smallest index such that $a_l<0$. Then $l$ is the minimizer of the sequence $p^{-i}(a_i-\frac{1}{p-1})$, and it is unique unless $p=2$, $a_l=-1$ and $a_{l+1}=-3$. Moreover, we always have $\lambda_l>r$.
\end{lemma}
\begin{proof}
For the sake of simplicity, we write $$a_i^*=p^{-i}\left(a_i-\frac{1}{p-1}\right).$$

We first prove that, for all $i\geq 0$ such that $\lambda_i\leq r$, we have $a_i>0$. If $r'\leq\lambda_i\leq r$, then $a_i=+\infty$ by definition. Let $I=\{i\geq 0\mid\lambda_i<r'\}$. For any two elements $i,j\in I$ such that $i<j$, \autoref{lePropertiesOfC} implies that $c_{\lambda_j}-c_{\lambda_i}\geq2(\lambda_j-\lambda_i)$. Therefore, we have
$$a_j-a_i=b_j-b_i=(c_{\lambda_j}-c_{\lambda_i})-(\lambda_j-\lambda_i)-(j-i)\geq0,$$
so $a_i$ is non-decreasing if $\lambda_i<r'$. It suffices to prove that $a_0>0$ whenever $\lambda_0<r'$ (so that $I$ is non-empty and $\min I=0$). Indeed, in this case we have $a_0=c_{\lambda_0}-\lambda_0\geq 2\lambda_0-\lambda_0>0$. Therefore, we have proved $\lambda_l>r$.

We proceed to show that $a_l^*\leq a_{l+i}^*$ for all $i>0$. Applying \autoref{lePropertiesOfB}, we see that $a_{i+1}=a_i-2$ whenever $\lambda_i>r$.

Now we conclude that
\begin{align*}
a_l^*-a_{l+i}^*&=p^{-{l+i}}\left(p^ia_l-a_{l+i}-\frac{p^i-1}{p-1}\right)\\
&\leq p^{-{l+i}}\left((p^i-1)a_l+2i-\frac{p^i-1}{p-1}\right)\\
&\leq p^{-{l+i}}\left(2i-p\frac{p^i-1}{p-1}\right)\\
&\leq 0.
\end{align*}
The last inequality holds since $p\geq 2$ and $\frac{p^i-1}{p-1}=1+p+\dots+p^{i-1}\geq i$. From this we see that $a_l^*\leq a_{l+i}^*$ for all $i>0$, and equality holds only if $a_l={-1}$, $i=1$ and $p=2$.

What remains is to prove $a_l^*<a_{l'}^{*}$ if $l'<l$. The definition of $l$ implies that $a_{l'}\geq0$ for all $l'<l$. Obviously we have $a_l^*<0<a_{l'}^{*}$ if $a_{l'}>0$, so we assume that $a_{l'}=0$ for some $l'<l$. If $\lambda_{l'}\leq r$, then $a_{l'}>0$, which is a contradiction. Thus we must have $\lambda_{l'}\geq r+1$, so $0=a_{l'}=a_{l}+2(l-l')$. By the minimality of $l$, the only possibility is that $l=l'+1$ and $a_{l}=-2$, but in this case $a_l^*<a_{l'}^{*}$ holds as well.
\end{proof}

Now we are in the position to state our main results. The first theorem corresponds to Case~\ref{caseLambda1} in Lemma~\ref{lePropertiesOfLambda}.

\begin{theorem}\label{thDivHomMain1}
Let $p$ be a prime, $G=\prod_{i=1}^{r}C_{p^i}^{k_i}$ be a finite Abelian $p$-group, and $q$ be a prime power such that $p\nmid q$. Suppose that $G=C_{p^r}\times G_0$, where $|G_0|=p^{s}$. Let $c_i$ and $\lambda_i$ be as defined in \autoref{thFormulaOfLogF}. Suppose that either $p\geq 3$ or $\lambda_0\geq2$. Then there exist integers $l$ and $a_l$ such that
\begin{align*}
v_p(\#\operatorname{Hom}(G,GL_n(\Fq)))&\geq \lambda_0\left\lfloor \frac{n}{d}\right\rfloor+\sum_{i=1}^{l}\left\lfloor \frac{n}{dp^i}\right\rfloor+a_l\left\lfloor \frac{n}{dp^l}\right\rfloor,
\end{align*}
where $d$ is the order of $q$ modulo $p$. Equality holds whenever $dp^l\mid n$, except if $p=2$ and $a_l=-1$.

The precise values of $l$ and $a_l$ are given by:
\begin{itemize}
\item If $\lambda_0>r+s$, then $l=0$ and $a_l=r+s-\lambda_0$.
\item If $r-s+2\leq \lambda_0\leq r+s$ and $r+s-\lambda_0$ is even, then $l=\frac{r+s-\lambda_0+2}{2}$ and $a_l=-2$.
\item If $r-s+2\leq \lambda_0\leq r+s$ and $r+s-\lambda_0$ is odd, then $l=\frac{r+s-\lambda_0+1}{2}$ and $a_l=-1$.
\item If $\lambda_0\leq r-s+1$, then $l=r-\lambda_0+1$ and $a_l=-r+s+\lambda_0-2$.
\end{itemize}
\end{theorem}

\begin{proof}
We use Theorem~\ref{thDivFMain} in the decomposition in \eqref{eqFormulaOfLogFAlt}, to conclude that
\begin{equation}
v_p([z^n]F(z))\geq a_l\left\lfloor\frac{n}{p^l}\right\rfloor-v_p\left(\left\lfloor\frac{n}{p^l}\right\rfloor!\right),
\end{equation}
where $l={\arg\min}_{i\geq 0}p^{-i}(a_i-\frac{1}{p-1})$.

Then using the arguments in \autoref{thDivHomLowerBound1} with $a_l$ instead of $b_l$ gives an explicit lower bound for $\#\hom(G,GL_n(\Fq))$. It remains to provide explicit ways to calculate $l$ and $a_l$ by using \autoref{leFormMinimizerOfAL}.

Let $i_0$ be the smallest index such that $\lambda_{i_0}>r$. Whenever $i\geq i_0$, \autoref{lePropertiesOfC} yields $c_{\lambda_i}=r+s$, and \autoref{lePropertiesOfLambda} yields $\lambda_i=\lambda_0+i$. This means that $a_i=r+s-\lambda_0-2i$ for all $i\geq i_0$.

Now \autoref{leFormMinimizerOfAL} yields that $l=\min\{i\geq i_0\mid a_i<0\}$. We distinguish several cases here.

\begin{enumerate}
   \item $\lambda_0>r+s$. In this case, $i_0=0$, and $a_{i_0}=r+s-\lambda_0<0$,  so we have $l=0$ and $a_l=r+s-\lambda_0$.
   \item $r<\lambda_0\leq r+s$. In this case, we still have $i_0=0$, but $a_{i_0}=r+s-\lambda_0\geq0$. Since $a_i=a_0-2i$ in this case, we conclude that $l=1+\floor{\frac{r+s-\lambda_0}{2}}$, and $a_l=-1$ or $-2$ when $r+s-\lambda_0$ is odd or even, respectively.
   \item $r-s+2\leq\lambda_0\leq r$. In this case, $i_0=r-\lambda_0+1$, and we know that $a_{i_0}=r+s-\lambda_0-2i_0=-r+s+\lambda_0-2\geq0$. Similar to above, we conclude that $l=i_0+1+\floor{\frac{-r+s+\lambda_0-2}{2}}=1+\floor{\frac{r+s-\lambda_0}{2}}$, and $a_l=-1$ or $-2$ when $r+s-\lambda_0$ is odd or even, respectively.
   \item $\lambda_0\leq r-s+1$. In this case, $i_0=r-\lambda_0+1$, and $a_{i_0}=r+s-\lambda_0-2i_0=-r+s+\lambda_0-2<0$, so we have $l=i_0$ and $a_l=-r+s+\lambda_0-2$.
\end{enumerate}

Finally, Lemma~\ref{leFormMinimizerOfAL} and Theorem~\ref{thDivFMain} imply that the minimizer is unique (so the bound will be periodically tight) unless $a_l=-1$ and $p=2$.
\end{proof}

The following theorem deals with Case~\ref{caseLambda2} in Lemma~\ref{lePropertiesOfLambda}.
\begin{theorem}\label{thDivHomMain2}
Let $G=\prod_{i=1}^{r}C_{2^i}^{k_i}$, and $q\equiv3\pmod{4}$ be a prime power. Suppose that $G=C_{2^r}\times G_0$, where $|G_0|=2^{s}$. Let $c_i$ and $\lambda_i$ be as defined in \autoref{thFormulaOfLogF}. Then there exist integers $l$ and $a_l$ such that
\begin{align*}
v_p(\#\operatorname{Hom}(G,GL_n(\Fq)))&\geq n+(\lambda_1-1)\left\lfloor \frac{n}{2}\right\rfloor+\sum_{i=2}^{l}\left\lfloor \frac{n}{2^i}\right\rfloor+a\left\lfloor \frac{n}{2^l}\right\rfloor.
\end{align*}
Equality holds whenever $2^l\mid n$, except if $a_l=-1$.

The precise values of $l$ and $a_l$ are given by:
\begin{itemize}
\item If $\lambda_1>r+s-1$, then $l=1$ and $a_l=r+s-\lambda_1-1$.
\item If $r-s+3\leq \lambda_1\leq r+s-1$ and $r+s-\lambda_1$ is odd, then $l=\frac{r+s-\lambda_1+1}{2}$ and $a_l=-2$.
\item If $r-s+3\leq \lambda_1\leq r+s-1$ and $r+s-\lambda_1$ is even, then $l=\frac{r+s-\lambda_1}{2}$ and $a_l=-1$.
\item If $\lambda_1\leq r-s+2$, then $l=r-\lambda_1+2$ and $a_l=-r+s+\lambda_1-3$.
\end{itemize}
\end{theorem}

\begin{proof}
The proof is completely analogous to \autoref{thDivHomMain1}, except for the fact that $\lambda_i=\lambda_1+i-1$ whenever $i\geq1$.
\end{proof}

\hrulefill

\subsection{Special case where $p=2$}\label{ssSpecialCase}
It remains to deal with the case where $p=2$ and $a_l=-1$ in Theorems~\ref{thDivHomMain1} and~\ref{thDivHomMain2}. Numerical results suggest that the $p$-divisibility of $\#\hom(G,GL_n(\Fq))$ is higher than the value given in these theorems by an amount of $\floor{n/2^{l+1}}-\floor{n/2^{l+2}}$. We will give a proof of this fact in this subsection.

We first prove a lemma concerning the $2$-adic properties of a certain class of power series.
\begin{lemma}\label{leDivSeriesSpecial2}
Let $u,v,w\in\Q_2$ such that $v_2(u)=-1$, and $v_2(v)=v_2(w)=-3$. Then, for all $n\geq1$, we have
\begin{equation*}
v_2\left([z^n]\exp(uz+vz^2+wz^4)\right)\geq -n-2\floor{\frac{n}{4}}-v_2\left(\floor{\frac{n}{4}}!\right),
\end{equation*}
and equality holds if $4\mid n$.
\end{lemma}
\begin{proof}
We substitute $z=2y$, so that $uz+vz^2+wz^4\in(y+y^2/2)+\Z_2[[y]]$. The lemma is an immediate consequence of the result by Chowla, Herstein and Moore \cite{MR0041849} which states that 
\[
v_2\left([y^n]\exp\left(y+\frac{y^2}{2}\right)\right)\geq -n-2\floor{\frac{n}{4}}-v_2\left(\floor{\frac{n}{4}}!\right).
\]
\end{proof}

\begin{theorem}\label{thDivHomMain3}
Let $G=\prod_{i=1}^{r}C_{2^i}^{k_i}$, and $q$ be an odd prime power. Suppose that the quantity $a_l$, as given by Theorems \ref{thDivHomMain1} or \ref{thDivHomMain2}, is equal to $-1$. Then we have
\begin{align*}
v_2(\#\operatorname{Hom}(G,GL_n(\Fq)))&\geq \lambda_0n+\sum_{i=1}^{l-1}\floor{\frac{n}{2^i}}+\floor{\frac{n}{2^{l+1}}}-\floor{\frac{n}{2^{l+2}}}
\end{align*}
when $q\equiv1\pmod{4}$, and
\begin{align*}
v_2(\#\operatorname{Hom}(G,GL_n(\Fq)))&\geq n+(\lambda_1-1)\floor{\frac{n}{2}}+\sum_{i=2}^{l-1}\floor{\frac{n}{2^i}}+\floor{\frac{n}{2^{l+1}}}-\floor{\frac{n}{2^{l+2}}}
\end{align*}
when $q\equiv3\pmod{4}$.
In both cases, equality holds whenever $2^{l+2}\mid n$.
\end{theorem}
\begin{proof}
Let us look again at the decomposition \ref{eqFormulaOfLogFAlt}, written more explicitly as 
\begin{equation*}
\log F(G,q,z)=\log F^\dag(z)+\sum_{i=0}^{\infty}\log F_i^*(z),
\end{equation*}
where
\begin{multline*}
F^\dag(z)=\exp\left(p^s\lim_{k\to\infty}\log F(C_{p^k},q;z)-(p^s-1)\left(h(q,z)-\frac{1}{d}h(q^d,z^d)\right)\right)\\\in \Zp[[z]]
\end{multline*}
and 
\begin{multline*}
\log F_i^*(z)=\frac{1}{d}\left(p^{c_{\lambda_i}-i}-p^{\lambda_i+s-i}\right)\left(h(q^{dp^i},z^{dp^i})-\frac{1}{p}h(q^{dp^{i+1}},z^{dp^{i+1}})\right)\\\in p^{a_i}\Zp[[z^{dp^i}]].
\end{multline*}

We know that the terms $\log F_i^*(z)$ with $i=l,l+1$ contribute to the apparent discrepancy. More explicitly, taking into account that $p=2$, $a_l=-1$ and $a_{l+1}=-3$, we know that
\begin{align*}
\log F_l^*(z)&\in 2^{-1}\Zp[[z^{2^l}]],\\
\log F_{l+1}^*(z)&\in 2^{-3}\Zp[[z^{2^{l+1}}]].
\end{align*}
Based on Theorem \ref{thDivLogFCase2}, the only ``relevant" part of these two power series are the first two terms in them, since the other terms all have higher divisibility. 

Now the power series containing only the first two terms of both $\log F_l^*(z)$ and $\log F_{l+1}^*(z)$ fit the form in Lemma~\ref{leDivSeriesSpecial2}, and the lemma yields that 
\[
v_2\left([z^{2^l}]F_l^*(z)F_{l+1}^*(z)\right)\geq -\floor{\frac{n}{2^l}}-2\floor{\frac{n}{2^{l+2}}}-v_2\left(\floor{\frac{n}{2^{l+2}}}!\right),
\]
for which equality holds if $2^{l+2}\mid n$. An argument analogous to the proof of Theorem \ref{thDivFMain} concludes the proof.
\end{proof}

\section{The modular case}
This section treats the modular case, where the characteristic of the underlying field $\Fq$ is the prime number $p$. In this case, we no longer have closed-form formulas like \autoref{thGenFun} for the generating function of the number $\#\hom(G,GL_n(\Fq))$ of homomorphisms from the finite Abelian $p$-group $G$ into $GL_n(\Fq)$ available.


We will restrict our attention to the case where $G$ is a cyclic group. 

Suppose that $G=C_{p^u}$ is a cyclic group, and $q=p^v$ for some prime $p$. We note that in this case $\#\hom(G,GL_n(\Fq))$ is just the number of elements with order dividing $p^u$ in $GL_n(\Fq)$.

In order to solve the above counting problem, we establish a bijection between the sets $\{A\in GL_n(\Fq)\mid A^{p^u}=I\}$ and $\{B\in M_n(\Fq)\mid B^{p^u}=0\}$ by observing that $A^{p^u}-I=(A-I)^{p^u}$ in characteristic~$p$. If we denote the number of $n\times n$ nilpotent matrices $B$ over $\Fq$ such that $B^{k}=0$ by $a_{n,k}$, then our problem is now reduced to calculating $v_p(a_{n,p^u})$.

We first count the nilpotent matrices by partitioning them into conjugacy classes, and giving a formula (Lemma~\ref{leConjClassSize}) for the size of each class. Then we prove a recurrence for the quantities $a_{n,k}$ (Theorem~\ref{thRecurrenceModular}), and use it to obtain a lower bound for $v_p(a_{n,k})$ that is quadratic in $n$ (Theorem~\ref{thDivModularMain}). This is then translated into the sought-for lower bound on $v_p(\#\hom(C_{p^u},GL_n(\Fq)))$ (Theorem~\ref{thDivModularFinal}).

\medskip
It is well-known that every matrix for which all its eigenvalues lie in the base field is conjugate to a matrix in Jordan normal form. The hypothesis is obviously satisfied for nilpotent matrices (since all the eigenvalues are zero), so we conclude that the conjugacy classes of $n\times n$ nilpotent matrices can be indexed by partitions of $n$, and the representative for such a class corresponding to a partition $\lambda=(\lambda_1,\lambda_2,\dots)$ can be chosen as a block diagonal matrix formed by Jordan blocks with eigenvalue $0$, and the sizes of the blocks are given by $\lambda_i$, $i=1,2,\dots$. In the following, we write $\lambda\vdash n$ if $\lambda$ is a partition of~$n$, that is, if $\lambda_1+\lambda_2+\cdots=n$.

We make use of a formula by Fine and Herstein \cite{fine1958} for the the size of the conjugacy class corresponding to a partition $\lambda\vdash n$.
\begin{lemma}[{\cite{fine1958}}]\label{leConjClassSize}
Let $\lambda\vdash n$ be a partition of $n$. Then the size of the conjugacy class corresponding to $\lambda$ is given by
\begin{equation*}
c(\lambda)=\frac{(-1)^{n-\sum a_i}q^{\beta(\lambda)}(q;q)_n}{\prod^{n}_{j=1}(q;q)_{a_i}},
\end{equation*}
where $\lambda$ is the partition consisting of $a_i$ parts~$i$, $i=1,2,\dots,n$, so that $n=a_1+2a_2+3a_3+\dots+na_n$, and $$\beta(\lambda)=\binom{n}{2}+\sum_{j=1}^{n}\binom{a_i+1}{2}-(a_1+a_2+\dots+a_n)^2-(a_2+a_3+\dots+a_n)^2-\dots-a_n^2.$$
\end{lemma}
%
%
%
%

Next we note that the index of nilpotency (the least integer $k$ such that $B^k=0$) of a nilpotent matrix is equal to the largest size of Jordan blocks in its Jordan normal form. Therefore, $a_{n,k}$ is equal to the sum of $c(\lambda)$ over all partitions $\lambda\vdash n$ with no parts larger than $k$.

Note that the size of every conjugacy class is a polynomial in $q$, and therefore so is $a_{n,k}$. Since $q=p^v$, we will consider the $q$-divisibility of $a_{n,k}$ instead of the $p$-divisibility (they differ by a factor of $v$). It is obvious from \autoref{leConjClassSize} that $v_q(c(\lambda))=\beta(\lambda)$. From this, one would be tempted to conclude that $v_q(a_{n,k})\geq \min_{\lambda\vdash n, \lambda_1\leq k}\beta(\lambda)$, but this minimum is $0$ since $\beta(1+1+\dots+1)=0$. Instead, we will exploit a recursive structure, made precise in the following theorem.
\begin{theorem}\label{thRecurrenceModular}
Let $a_{n,k}=\#\{B\in \Fq^{n\times n}\mid B^k=0\}$. Then for all non-negative integers $n,k$, we have
$$a_{n,k}=q^{n^2-n}-\sum_{m=k+1}^{n}\sum_{l=1}^{\lfloor n/m \rfloor}\frac{(-1)^{l(m-1)}(q;q)_n}{(q;q)_{n-lm}(q;q)_{l}}q^{f(l,m,n)}a_{n-lm,m-1},$$
where $f(l,m,n)=ln(m-2)+\binom{l(m-1)}{2}.$
\end{theorem}
\begin{proof}
We first use the famous result that the number of $n \times n$ nilpotent matrices over $\Fq$ is $q^{n^2-n}$ (for a proof, see \cite{Crabb06}) to obtain the formula
\begin{equation}
a_{n,k}=q^{n^2-n}-\sum_{\lambda\vdash n,\, \lambda_1>k}c(\lambda).
\end{equation}
Then we partition the set $\{\lambda\vdash n\mid\lambda_1>k\}$ according to the the largest part $\lambda_1$ and the number of such parts.
Let $\lambda\vdash n$ with $\lambda_1=\lambda_2=\dots=\lambda_l=m$ and $\lambda_{l+1}<m$. Furthermore, let $\mu\vdash(n-lm)$ be the partition resulting from removing all the $m$'s from $\lambda$. Then we have
\begin{equation*}
\frac{c(\lambda)}{c(\mu)}=\frac{(-1)^{l(m-1)}(q;q)_n}{(q;q)_{n-lm}(q;q)_{l}}q^{f(l,m,n)},
\end{equation*}
which is a direct consequence of \autoref{leConjClassSize}.

The set of partitions $\lambda\vdash n$, where $\lambda_1=\lambda_2=\dots=\lambda_l=m$ and $\lambda_{l+1}<m$, is in bijection with partitions of $\mu\vdash n-lm$, where $\mu_1<m$. From the above relation and the definition of $a_{n,k}$, we see that
\begin{align*}
a_{n,k}&=q^{n^2-n}-\sum_{\lambda\vdash n,\, \lambda_1>k}c(\lambda)\\
&=q^{n^2-n}-\sum_{m=k+1}^{n}\sum_{l=1}^{\lfloor n/m \rfloor}\sum_{\lambda\vdash n,\, \lambda_1=\dots=\lambda_l=m,\lambda_{l+1}<m}c(\lambda)\\
&=q^{n^2-n}-\sum_{m=k+1}^{n}\sum_{l=1}^{\lfloor n/m \rfloor}\frac{(-1)^{l(m-1)}(q;q)_n}{(q;q)_{n-lm}(q;q)_{l}}q^{f(l,m,n)}\sum_{\mu\vdash n-lm,\, \mu_1<m}c(\mu)\\
&=q^{n^2-n}-\sum_{m=k+1}^{n}\sum_{l=1}^{\lfloor n/m \rfloor}\frac{(-1)^{l(m-1)}(q;q)_n}{(q;q)_{n-lm}(q;q)_{l}}q^{f(l,m,n)}a_{n-lm,m-1}.
\end{align*}
\end{proof}

This recurrence enables us to prove a quadratic lower bound for $v_q(a_{n,k})$ by induction on $k$.

\begin{theorem}\label{thDivModularMain}
Let $a_{n,k}$ be defined as in Theorem \ref{thRecurrenceModular}. Then,
$$v_q(a_{n,k})\geq \frac{k-1}{k+1}\binom{n}{2},$$
and equality holds whenever $n\equiv 0,1\pmod{k+1}$.
\end{theorem}

\begin{proof}
We proceed by induction on $k$. The claim is obvious for $k=1$.
For accomplishing the induction step, we proceed as follows:
\begin{align*}
v_q(a_{n,k})&\geq \min_{m>k,\,lm\leq n}\left(f(l,m,n)+v_q(a_{n-lm,m-1})\right)\\
&\geq \min_{m>k,\,lm\leq n}\left(f(l,m,n)+\frac{m-2}{m}\binom{n-lm}{2}\right)\\
&=\min_{m>k,\,lm\leq n}\left(\frac{m-2}{m}\binom{n}{2}+\binom{l}{2}\right)\\
&=\frac{k-1}{k+1}\binom{n}{2}.
\end{align*}

To prove the claim about equality, we first observe that equality holds for $n=0$ or $n=1$. Next we claim that, if $v_q(a_{n-k-1,k})=\frac{k-1}{k+1}\binom{n-k-1}{2}$, then $v_q(a_{n,k})=\frac{k-1}{k+1}\binom{n}{2}$. To prove this, we first notice that $\frac{m-2}{m}\binom{n}{2}+\binom{l}{2}\geq \frac{k-1}{k+1}\binom{n}{2}$, and equality occurs if and only if $l=1$ and $m=k+1$. Thus, by the induction hypothesis, we know that
$$f(l,m,n)+v_q(a_{n-lm,m-1})=\frac{m-2}{m}\binom{n}{2}+\binom{l}{2}=\frac{k-1}{k+1}\binom{n}{2}$$
if $(l,m)=(1,k+1)$, and
$$f(l,m,n)+v_q(a_{n-lm,m-1})\geq \frac{m-2}{m}\binom{n}{2}+\binom{l}{2}>\frac{k-1}{k+1}\binom{n}{2}$$
if $(l,m)\neq(1,k+1)$. Thus, we have $v_q(a_{n,k})=\frac{k-1}{k+1}\binom{n}{2}$ as we claimed.
\end{proof}

Translating this result back, we have the following lower bound for\break
$v_p(\#\hom(C_{p^u},GL_n(\FF{p^v})))$.
\begin{theorem}\label{thDivModularFinal}
For any prime $p$ and any non-negative integers $u,v$, we have
\begin{equation}
v_p(\#\hom(C_{p^u},GL_n(\FF{p^v})))\geq v\frac{p^u-1}{p^u+1}\binom{n}{2}.
\end{equation}
Equality holds if $n\equiv 0,1\pmod{p^u+1}$.
\end{theorem}

\section{Concluding remarks}
Let $G$ be an Abelian $p$-group. We have established lower bounds for the $p$-adic valuation of $\#\hom(G,GL_n(\Fq))$ where either $p\nmid q$ or $p\mid q$ and $G$ is cyclic. The natural remaining part of the problem deals with the case $p\mid q$ and $G$ non-cyclic. If we attempt to use similar arguments as in Section~3, we will face the problem of counting commuting tuples $(B_1,B_2,\dots)$ of nilpotent matrices, where the index of nilpotency for each $B_i$ has different upper bounds. To this day, there is no method for counting these tuples. One hopes that a recurrence similar to Theorem \ref{thRecurrenceModular} could be found.

On a more general perspective, the results obtained in this paper enable us to formulate a conjecture concerning $v_p(\#\hom(G,H))$ for general finite groups $G$ and $H$. The results in this paper have an interesting analogy with the fact that $v_p(|GL_n(\Fq)|)=Cn+O(1)$ when $p\nmid q$, and $v_p(|GL_n(\Fq)|)=Cn^2+O(n)$ when $p\mid q$. This leads us to conjecture that $v_p(\#\hom(G,H))\geq(C_G+o(1))v_p(|H|)$ for all finite groups $G$ and $H$, where $C_G$ is a constant depending on $G$. A closer inspection disproves this conjecture: we simply take $G=C_p$ and $H=C_{p^k}$, and let $k\to\infty$. In view of this counterexample, we replace $v_p(|H|)$ by another quantity related to $H$, and propose the following conjecture.
\begin{conj}
Let $G$ be an arbitrary finite group, and $p$ be a prime dividing $|G|$. Then there exist constants $0<c_G<C_G$, depending only on $G$, such that
\begin{equation}
c_G\rk_p(H)\leq v_p(\#\hom(G,H))\leq C_G\rk_p(H)
\end{equation}
for all finite groups $H$. Here $\rk_p(H)$ is the \emph{$p$-rank} of $H$, defined as the maximum rank of elementary Abelian $p$-subgroups in $H$.
\end{conj}

\appendix

\section*{Appendix A: Logarithm of the generating function and the proof of Theorem \ref{thLogOfGeneratingFunction}}

\setcounter{equation}{0}
\setcounter{theorem}{0}
\global\def\thesection{\mbox{A}}

The purpose of this appendix is to prove all the auxiliary theorems in Subsection~\ref{ssAux}.

\begin{lemma}
Let $g(q,z)=\frac{f(q,qz)}{f(q,z)}$. The power series $g(z)$ satisfies the equation
\begin{equation}\label{eqGFunctionEquation}
g(q,z)g(q,z/q)=g(q,z/q)+z,
\end{equation}
and we have the following expression for $g$:
\begin{equation}
g(q,z)=1+\sum_{n\geq1}(-1)^{n-1}q^{-\binom{n}{2}}C_{n-1}(q)z^n,
\end{equation}
where $C_n(q)$ is Carlitz and Riordan's $q$-Catalan number (see, for example, {\em\cite{carlitz1964}}).
\end{lemma}
\begin{proof}
From
\begin{multline*}
$$[z^n]f(q,qz)-[z^n]f(q,z)-[z^{n-1}]f(q,z/q)\\
=\frac{(-1)^n}{q^{\binom{n}{2}-n}(q;q)_n}-\frac{(-1)^n}{q^{\binom{n}{2}}(q;q)_n}-\frac{(-1)^{n-1}}{q^{\binom{n-1}{2}+(n-1)}(q;q)_{n-1}}=0,
\end{multline*}
we see that $f(q,qz)=f(q,z)+zf(q,z/q)$. Consequently, we have
$$g(q,qz)g(q,z)=\frac{f(q,q^2z)}{f(q,z)}=\frac{f(q,qz)}{f(q,z)}+qz=g(q,z)+qz.$$
This proves the functional equation for $g(z)$.

Carlitz and Riordan's $q$-Catalan number is defined by $C_0(q)=1$, and
$$C_{n+1}(q)=\sum_{m=0}^{n}q^{(n-m)(m+1)}C_{m}(q)C_{n-m}(q).$$
Let $\tilde{C}_n(q)=(-1)^nq^{\binom{n+1}{2}}[z^{n+1}]g(q,z)$, We will
prove $C_n=\tilde{C}_n$ by showing that they satisfy the same recurrence. In fact, for all $n\geq0$, we have
\begin{align*}
0&=-[z^{n+2}]g(q,z/q)+[z^{n+2}]g(q,z)g(q,z/q)\\
&=-[z^{n+2}]g(q,z/q)+\sum_{m=0}^{n+2}[z^{n+2-m}]g(q,z)[z^{m}]g(q,z/q)\\
&=-q^{-n-2}[z^{n+2}]g(q,z)+\sum_{m=0}^{n+2}q^{-m}[z^{n+2-m}]g(q,z)[z^{m}]g(q,z)\\
&=[z^{n+2}]g(q,z)+\sum_{m=1}^{n+1}q^{-m}[z^{n+2-m}]g(q,z)[z^{m}]g(q,z)\\
&=(-1)^{n+1}\tilde{C}_{n+1}(q)q^{-\binom{n+2}{2}}\\
&\kern1cm
+\sum_{m=1}^{n+1}(-1)^{n}q^{-\binom{n+2-m}{2}-\binom{m}{2}-m}\tilde{C}_{n+1-m}(q)\tilde{C}_{m-1}(q)\\
&=(-1)^{n-1}q^{-\binom{n+2}{2}}\left(\tilde{C}_{n+1}(q)\vphantom{\sum_{m=0}^{n}}\right.\\
&\kern1cm
\left.
-\sum_{m=0}^{n}(-1)^{n}q^{\binom{n+2}{2}-\binom{n+1-m}{2}+\binom{m+1}{2}-m-1}\tilde{C}_{n-m}(q)\tilde{C}_{m}(q)\right)\\
&=(-1)^{n-1}q^{-\binom{n+2}{2}}\left(\tilde{C}_{n+1}(q)-\sum_{m=0}^{n}(-1)^{n}q^{(n-m)(m+1)}\tilde{C}_{n-m}(q)\tilde{C}_{m}(q)\right).
\end{align*}

and
\begin{align*}
0&=-1-[z^1]g(q,z/q)+[z^1]g(q,z)g(q,z/q)\\
&=-1-q^{-1}[z^1]g(q,z)+(1+q^{-1})[z^0]g(q,z)[z^1]g(q,z)\\
&=-1+[z^1]g(q,z)\\
&=\tilde{C}_0(q)-1,
\end{align*}
as desired.
\end{proof}

\begin{lemma}
Let $d$ be a positive integer, and $\omega$ be a primitive $d$-th root of unity. Then
\[
\prod^{d-1}_{i=0}g(\omega,\omega^i z)=\frac{1+\sqrt{1+4z^d}}{2}.
\]
\end{lemma}
\begin{proof}
In this proof, all subscripts are modulo~$d$.

Let $g_i(z)=g(\omega,\omega^i z)$. We use \ref{eqGFunctionEquation} to
get
$$g_{i}(z)(1-g_{i+1}(z))=-\omega^{i+1}z$$
for $0\leq i<d$. Forming the products of the left-hand and the
right-hand sides of these $d$ equations, we obtain
\begin{equation}
\prod_{i=0}^{d-1}g_i(z)(1-g_i(z))=-z^d.
\end{equation}

Note that the claim follows from a similar equation:
\begin{equation}
\left(\prod_{i=0}^{d-1}g_i(z)\right)
\left(1-\prod_{i=0}^{d-1}g_i(z)\right)\overset{?}{=}-z^d.
\end{equation}
Therefore, we need to prove that $\prod_{i=0}^{d-1}(1-g_i(z))=\left(1-\prod_{i=0}^{d-1}g_i(z)\right).$

Let $\mathcal{H}$ be the family of nonempty subsets of $\mathbb{Z}/d\mathbb{Z}$ that does not contain two elements that differ by $1$. For every such subset $S$, we define the function 
$$h_S(z)=\prod_{i\in S}g_i(z)(g_{i+1}(z)-1)=\sum_{S\subseteq T\subseteq S\cup(S+1)}(-1)^{|T|}\prod_{i\in T}g_i(z).$$

We claim that the sum of all such functions, $\sum_{S\in \mathcal{H}}h_S(z)$, is equal to
\begin{multline*}
\prod_{i=0}^{d-1}(1-g_i(z))-\left(1-\prod_{i=0}^{d-1}g_i(z)\right)\\
=\sum_{\emptyset\subsetneqq T\subsetneqq\mathbb{Z}/d\mathbb{Z}}(-1)^{|T|}\prod_{i\in
  T}g_i(z)+(1+(-1)^d)\prod_{i=0}^{d-1}g_i(z).
\end{multline*}
This claim is proved by noting that for every non-empty proper subset $T$ of $\mathbb{Z}/d\mathbb{Z}$, there exists a unique $S\in\mathcal{H}$ such that $S\subseteq T\subseteq S\cup(S+1)$; and, for $T=\mathbb{Z}/d\mathbb{Z}$, there exists no such $S$ if $d$ is odd, and two such $S$ if $d$ is even.

It remains to show that $\sum_{S\in \mathcal{H}}h_S(z)=0$. We break
this sum according to equivalence classes of $S$ by the translations
in $\mathbb{Z}/d\mathbb{Z}$. The sum for a class is given by
\begin{align*}
\sum_{j=0}^{d-1}h_{S+j}(z)&=\sum_{j=0}^{d-1}\prod_{i\in S}g_{i+j}(z)(g_{i+j+1}(z)-1)\\
&=\sum_{j=0}^{d-1}z^{|S|}\omega^{\sum_{i\in S}(i+j+1)}\\
&=z^{|S|}\omega^{\sum_{i\in S}(i+1)}\underbrace{\sum_{j=0}^{d-1}\omega^{j|S|}}_{1\leq|S|\leq \lfloor d/2\rfloor}\\
&=0.
\end{align*}
\end{proof}

We are now ready to prove Theorem \ref{thLogOfGeneratingFunction}.

\begin{proof}[Proof of Theorem \ref{thLogOfGeneratingFunction}]
We first note that
\begin{align*}
[z^n]\log g(q,z)&=[z^n]\log f(q,qz)-[z^n]\log f(q,z)=(q^n-1)[z^n]\log f(q,z).
\end{align*}

Let $P_n(q)=(-1)^{n-1}nq^{\binom{n}{2}}[z^n]\log g(q,z)$. We write

\begin{align*}
P_n(q)&=(-1)^{n-1}nq^{\binom{n}{2}}\sum_{m\geq 1}\frac{(-1)^{m-1}}{m}[z^n](g(q,z)-1)^m\\
&=\sum_{m\geq 1}q^{\binom{n}{2}}\frac{(-1)^{m+n}n}{m}\sum_{\substack{a_1+a_2+\dots+a_m=n\\a_i>0}}\prod_{j=i}^m[z^{a_i}]g(q,z)\\
&=\sum_{m\geq 1}\sum_{\substack{a_1+a_2+\dots+a_m=n\\a_i>0}}q^{\binom{n}{2}-\sum_i \binom{a_i}{2}}\frac{n}{m}\prod_{i=1}^mC_{a_i-1}(q)\\
&=\sum_{b_1+2b_2+\dots+nb_n=n}q^{\binom{n}{2}-\sum_j b_j\binom{j}{2}}\underbrace{\frac{n}{\sum b_j}\binom{\sum b_j}{b_1,b_2,\dots,b_n}}_{\in\mathbb{Z}}\prod_{j=1}^nC_{j-1}^{b_j}(q)\\
&\in\mathbb{Z}[q].
\end{align*}

The specific values of $P_n(q)$ that we need are:
\begin{align*}
P_{dn}(\omega_d)&=(-1)^{dn-1}dn \underbrace{\omega_d^{\binom{dn}{2}}}_{=(-1)^{n(dn-1)}}[z^{dn}]\log g(\omega_d,z)\\
&=(-1)^{(n-1)(dn-1)}dn\left(\frac{1}{d}\sum_{i=0}^{d-1}[z^{dn}]\log g(\omega_d,\omega_d^iz)\right)\\
&=(-1)^{(n-1)(dn-1)}n[z^{dn}]\log\prod_{i=0}^{d-1}g(\omega_d,\omega_d^iz)\\
&=(-1)^{n-1}n\underbrace{[z^{dn}]\log\frac{1+\sqrt{1+4z^d}}{2}}_{=(-1)^{n-1}\binom{2n-1}{n-1}/n}\\
&=\binom{2n-1}{n-1}.
\end{align*}

\end{proof}

\section*{Appendix B: Auxiliary divisibility results and the proof of Theorem \ref{thDivLogFCase1}}

\setcounter{equation}{0}
\setcounter{theorem}{0}
\global\def\thesection{\mbox{B}}

\begin{lemma}\label{leMoebiusDivisibility}
Let $f\in \mathbb{Z}[x_1,x_2,\dots]$, and $d$ be a positive integer. Then we have
\begin{equation}
\sum_{e\mid d}\mu(d/e)f(x_1^{d/e},x_2^{d/e},\dots)^e\in d\mathbb{Z}[[x]].
\end{equation}
\end{lemma}
\begin{proof}
The proof is divided into two parts with increasing generality.
\begin{enumerate}
\item We first prove the claim for the case where $f(x_1,x_2,\dots,x_n)=x_1+x_2+\dots+x_n$. In this case, let $C_d$ be a cyclic group with order $d$. We consider the set of maps $M=\{f:C_d\to\{x_1,x_2,\dots,x_n\}\}$, and let the action of $C_d$ on $M$ be given by $(g.f)(a)=f(g+a)$. In this case, the quantity
\[
(x_1+x_2+\cdots)^d
\]
is the generating function of the set $M$, while
\[
\left(x_1^{d/e}+x_2^{d/e}+\cdots\right)^e
\]
gives the generating function of the maps invariant under the action of a subgroup $C_e$. Therefore, by using M\"obius inversion, we conclude that
\begin{equation*}
\sum_{e\mid d}\mu(d/e)\left(x_1^{d/e}+x_2^{d/e}+\dots\right)^e
\end{equation*}
is the generating function of all the maps in $M$ with a trivial stabilizer. It's obvious that such maps can be partitioned into full orbits with size $d$, so their generating function is divisible by $d$.

\item Now suppose that $f\in \mathbb{Z}[x_1,x_2,\dots]$ is a polynomial with integer coefficients. By reducing $f$ modulo $d$, we suppose without loss of generality that the coefficients of $f$ belongs to $\{0,1,2,\dots,d-1\}$. We write $$f(x_1,x_2,\dots)=\sum_{i=1}^{n} g_i(x_1,x_2,\dots)$$ where each $g_i$ is a monomial with coefficient $1$. The condition on $g_i$ ensures that $g_i(x_1^k,x_2^k,\dots)=g_i^k(x_1,x_2,\dots)$ for all positive integer $k$. Consequently, we have
\begin{align*}
\sum_{e\mid d}\mu(d/e)f(x_1^{d/e},x_2^{d/e},\dots)^e
&=\sum_{e\mid d}\mu(d/e)\left(\sum_{i=1}^{n} g_i(x_1^{d/e},x_2^{d/e},\dots)\right)^e\\
&=\sum_{e\mid d}\mu(d/e)\left(\sum_{i=1}^{n} g_i^{d/e}(x_1,x_2,\dots)\right)^e.\\
\end{align*}
We claim that the last sum is divisible by $d$. This is proved by substituting $g_i(x_1,x_2,\dots)$ in place of $x_i$ in Case~1.
\end{enumerate}
\end{proof}

\begin{lemma}\label{leExpansionOfPn}
There exists a sequence of polynomials $Q_n(q)\in\mathbb{Z}[q]$ such that
\begin{equation}\label{eqExpansionOfPn}
P_n(q)=\sum_{d\mid n}(-1)^{n-d}d\frac{q^n-1}{q^{n/d}-1}q^{\binom{n}{2}-\frac{n}{d}\binom{d}{2}}Q_d(q^{n/d}).
\end{equation}
\end{lemma}

\begin{proof}
We consider the polynomials
\begin{equation}
R_n(q)=\sum_{d\mid n}(-1)^{n-d}\mu(n/d)q^{\binom{n}{2}-\frac{n}{d}\binom{d}{2}}P_d(q^{n/d}).
\end{equation}
By M\"obius inversion, the sequence of rational functions $Q_n(q)=\frac{1}{n}\frac{q-1}{q^n-1}R_n(q)$ will satisfy the equation in the lemma. It remains to prove that $Q_n(q)$ is actually a sequence of polynomials with integral coefficients; equivalently, we need to prove that, for every $n\geq1$, $R_n(q)$ is divisible by both $n$ and $\frac{q^n-1}{q-1}$.

The fact that $\frac{q^n-1}{q-1}\mid R_n(q)$ is relatively straightforward, as we only need to verify that $R_n(e^{2\pi i a/n})=0$ for every $1\leq a\leq n-1$. Indeed, we have
\begin{equation}
\begin{split}
R_n(e^{2\pi i a/n})&=\sum_{d\mid n}(-1)^{n-d}\mu(n/d)e^{\frac{2\pi i a}{n}(\binom{n}{2}-\frac{n}{d}\binom{d}{2})}P_d(e^{2\pi i a/d})\\
&=\sum_{d\mid n}(-1)^{(n-d)(a-1)}\mu(n/d)\binom{2(a,d)-1}{(a,d)-1}\\
&=\sum_{d\mid n}(-1)^{(n-n/d)(a-1)}\mu(d)\binom{2(a,n/d)-1}{(a,n/d)-1},
\end{split}
\end{equation}
where the second equality comes from Theorem \ref{thLogOfGeneratingFunction}.

Because of the presence of the term $\mu(d)$ in the summand, the sum in the last equation can instead be written as a sum over all square-free factors of $n$. For any prime factor $p$ of $n$, the set of such factors can be partitioned into pairs $\{d,pd\}$ such that $p\nmid d$.

Now, since $1\leq a\leq n-1$, there exists a prime $p$ satisfying $v_p(a)\leq v_p(n)-1$. This means that $\gcd(a,n/d)=\gcd(a,n/pd)$ for any $d$ not divisible by $p$. We use this prime $p$ to partition the set of square-free factors of $n$ as above. The contribution of each pair to the sum is
\begin{align*}
(-1)^{(n-n/d)(a-1)}&\mu(d)\binom{2(a,n/d)-1}{(a,n/d)-1}\\
&\kern2cm
+(-1)^{(n-n/pd)(a-1)}\mu(pd)\binom{2(a,n/pd)-1}{(a,n/pd)-1}\\
&\kern-10pt
=\binom{2(a,n/d)-1}{(a,n/d)-1}\mu(d)\left((-1)^{(a-1)(n-n/d)}-(-1)^{(a-1)(n-n/pd)}\right)\\
&\kern-10pt
=0.
\end{align*}
Therefore, $R_n(e^{2\pi i a/n})$=0 for $1\leq a\leq n-1$, which proves our claim. The last equality holds because
$$(a-1)(n-n/d)-(a-1)(n-n/pd)=(a-1)\frac{n(p-1)}{pd}$$
is always even. (This is obvious if $p$ is odd; if $p=2$, either $v_2(n)\geq 2$ so that $\frac{n}{2d}$ is even, or $v_2(n)=1$ so that $a$ is odd and $a-1$ is even.) Hence, the parity of the two exponents is the same.

We now proceed to prove that $\frac1n R_n(q)\in\Z[q]$. To this end, we substitute the expression of $P_n(q)$ in the proof of Theorem \ref{thLogOfGeneratingFunction} into the definition of $R_n(q)$ to get
\begin{multline}\label{eqRnSum}
\frac1nR_n(q)= \sum_{m=1}^{\infty}\sum_{d\mid n}\sum_{\substack{\sum_i a_i=d\\a_i>0}}(-1)^{n-d}\mu(n/d)q^{\binom{n}{2}-\frac{n}{d}\sum_i\binom{a_i}{2}}\frac{d}{mn}\prod_{i=1}^mC_{a_i-1}(q^{n/d}).
\end{multline}

We call a sequence $(a_1,a_2,\dots,a_m)$ \emph{primitive} if it is not fixed by any non-trivial cyclic permutations. It is obvious that a non-primitive sequence is the concatenation of some copies of a primitive sequence. Now let $(a_1,a_2,\dots,a_m)$ be a primitive sequence, where $\sum a_i=c$. The $e$-fold self-concatenation of the sequence $(a_i)$ contributes to the multi-sum in \ref{eqRnSum} if and only if $ce\mid n$. This means that the contribution of $(a_i)$ and its self-concatenations is given by
\begin{align*}
\frac{c}{mn}(-1)^nq^{\binom{n}{2}-\frac{n}{c}\sum_i\binom{a_i}{2}}\sum_{e\mid n/c}(-1)^{ce}\mu(n/ce)\prod_{i=1}^mC_{a_i-1}(q^{n/ce})^e.
\end{align*}
We use Lemma~\ref{leMoebiusDivisibility} with $d=n/c$ and $f(q)=(-1)^c\prod_{i=1}^m C_{a_i-1}(q)$ to prove that the last sum is divisible by $n/c$, so that the contribution of $(a_i)$ belongs to $\frac1m \Z[q]$. The $m$ in the denominator is eliminated by summing over all $m$ cyclic permutations of a primitive sequence.

\end{proof}

\begin{proof}[Proof of Theorem \ref{thDivLogFCase1}]
	Let $Q_n(q)$ be defined as in Lemma~\ref{leExpansionOfPn}, and define
	\[
	k(q,z)=\sum_{d\geq 1}\frac{(-1)^{d-1}Q_d(q)}{q^{\binom{d}{2}}(q-1)}z^n.
	\]
	Combining Theorem \ref{thLogOfGeneratingFunction} and Lemma~\ref{leExpansionOfPn}, we obtain the identity
	\[
	h(q,z)=\sum_{n=1}^{\infty}\frac{k(q^n,z^n)}{n}.
        \]
	So we conclude that
	\begin{equation}\label{eqHAsASum}
	h(q,z)-\frac1p h(q^p,z^p)=\sum_{\substack{n\geq1\\p\nmid n}}\frac{k(q^n,z^n)}{n}.
	\end{equation}
	The definition of $k$ implies that $k(q,z)\in p^{-v_p(q-1)}\Zp[z]$, and consequently, for all $n$ such that $p\nmid n$, we have
	\[
	\frac{k(q^n,z^n)}{n}\in p^{-v_p(q^n-1)}\Zp[z].
	\]
	Now Lemma~\ref{leDivQNMinus1} yields $v_p(q^n-1)=v_p(q-1)$ when $p\nmid n$, so the theorem is proved.
\end{proof}

\setcounter{equation}{0}
\setcounter{theorem}{0}
\global\def\thesection{\mbox{C}}

\section*{Appendix C: Proof of Theorem \ref{thDivLogFCase2}}

We begin this section with some preliminary lemmas on the $2$-adic properties of harmonic numbers and binomial coefficients, and use them to obtain a parity result for the sequence $Q_n(q)$ of polynomials defined in Lemma~\ref{leExpansionOfPn}.

\begin{lemma}\label{lev2Harmonic}
Let $d\in\Z^+$. Then we have
\[
v_2\left(\sum_{i=1}^{d}\frac{1}{2i-1}\right)=2v_2(d).
\]
\end{lemma}
\begin{proof}
We have the identity 
\begin{align*} 
\sum_{i=1}^{d}\frac{1}{2i-1}&=\frac12\sum_{i=1}^{d}\left(\frac{1}{2i-1}+\frac{1}{2d-2i+1}\right)\\
&=d\sum_{i=1}^{d}\frac{1}{(2i-1)(2d-2i+1)},
\end{align*}
so it suffices to prove that the $2$-divisibility of $\sum_{i=1}^{d}\frac{1}{(2i-1)(2d-2i+1)}$ is exactly $v_2(d)$.

Suppose that $d=2^km$ where $m$ is odd. For $1\leq j\leq m$, we consider the sum
\[
\sum_{i=2^kj+1}^{2^kj+2^k}\frac{1}{(2i-1)(2d-2i+1)}\equiv\sum_{i=1}^{2^k}\frac{-1}{(2i-1)^2}\pmod{2^{k+1}}.
\]
When $i$ ranges over all residue classes modulo $2^k$, $(2i-1)^{-1}$ ranges over all odd residue classes modulo $2^{k+1}$. Therefore, we conclude that 
\begin{align*} 
\sum_{i=1}^{2^k}\frac{-1}{(2i-1)^2}&\equiv-\sum_{i=1}^{2^k}(2i-1)^2\\
&=-\frac{1}{3}\left(2^{3k+2}-2^{k}\right)\\
&\equiv2^k\pmod{2^{k+1}}.
\end{align*}
Thus $\sum_{i=1}^{d}\frac{1}{(2i-1)(2d-2i+1)}$ is the sum of $m$ terms where each term is $2^k\pmod{2^{k+1}}$, so its $2$ divisibility is $k$ as claimed.
\end{proof}

\begin{lemma}\label{lev2Binomial}
Let $d\in\Z^+$. Then we have
\[
v_2\left(\binom{4d-1}{2d-1}-(-1)^d\binom{2d-1}{d-1}\right)=2+2v_2(d)+s_2(d-1),
\]
where $s_2(\cdot)$ represents the base-2 digit sum.
\end{lemma}
\begin{proof}
Simple algebraic manipulations lead to the identity
\begin{equation}
\binom{4d-1}{2d-1}-(-1)^d\binom{2d-1}{d-1}=\frac{2d-1}{d}\binom{2d-2}{d-1}\left(\prod_{i=1}^{d}\left(\frac{4d}{2i-1}-1\right)-(-1)^{d}\right).
\end{equation}
Using the fact that $v_2\left(\binom{2d-2}{d-1}\right)=s_2(d-1)$, it suffices to prove that the $2$-divisibility of the last factor is equal to $3v_2(d)+2$. To this end, we consider the polynomial $P(x)=\prod_{i=1}^{d}\left(\frac{x}{2i-1}-1\right)-(-1)^{d}$. We have to prove that $v_2(P(4d))=3v_2(d)+2$.

It is obvious that 
\[
P(x)=(-1)^d\left(-h_1 x+h_2x^2\right)+O(x^3),
\]
where $h_1=\sum_{i=1}^{d}\frac{1}{2i-1}$, $h_2=\sum_{1\leq i<j\leq d}\frac{1}{(2i-1)(2j-1)}$, and the remaining coefficients all belong to $\Z_2$.

Lemma~\ref{lev2Harmonic} states that $v_2(h_1)=2v_2(d)$, so we have $v_2(4dh_1)=3v_2(d)+2$, as desired. Since $v_2((4d)^3)>3v_2(d)+2$, it remains to prove $v_2(h_2)>v_2(d)-2$, so that $v_2((4d)^2h_2)>3v_2(d)+2$. This claim is obvious if $d$ is odd. If $d$ is even, then we have the identity
\begin{align*}
h_2&=\sum_{i=1}^{d/2}\frac{1}{(2i-1)(2(d+1-i)-1)}+
\underset{ i+j\neq d+1}{\sum_{1\leq i<j\leq d}}\frac{1}{(2i-1)(2j-1)}\\
&=\frac{1}{2d}h_1+\sum_{1\leq i<j\leq d/2}\frac{4d^2}{(2i-1)(2j-1)(2d-2i+1)(2d-2j+1)}.
\end{align*}
Therefore we have 
$$v_2(h_2)\geq\min(-1-v_2(d)+v_2(h_1),2+2v_2(d))=v_2(d)-1,$$ 
as claimed.

\end{proof}

\begin{lemma}\label{leQn1Parity}
Let $n$ be a positive integer, and $Q_n(q)$ be as defined in Lemma~\ref{leExpansionOfPn}. Then $Q_n(1)$ is odd if and only if $n$ is square-free.
\end{lemma}
\begin{proof}
By Theorem \ref{thLogOfGeneratingFunction}, we have $P_n(1)=\binom{2n-1}{n-1}$. In view of this fact, the definition of $Q_n$ implies the formula
\begin{equation}
Q_n(1)=\frac{1}{n^2}\sum_{d\mid n}(-1)^{n-d}\mu(n/d)\binom{2d-1}{d-1}.
\end{equation}
Here we also state the fact that $\binom{2n-1}{n-1}$ is odd if and only if $n$ is a power of $2$.

We now divide the proof into two parts according to the parity of $n$.
\begin{enumerate}
\item \emph{$n$ is odd.} In this case, we only need to look at the parity of the numerator of $Q_n(1)$, namely the sum $\sum_{d\mid n}(-1)^{n-d}\mu(n/d)\binom{2d-1}{d-1}$. 
Every term in this sum is even unless $d$ is a power of $2$ and $n/d$ is square-free. Since $n$ is odd, this can only happen when $d$=1 and $n$ is square-free.
\item \emph{$n$ is even.} Here we estimate the quantity 
\[
v_2\left(\sum_{d\mid n}(-1)^{n-d}\mu(n/d)\binom{2d-1}{d-1}\right),
\]
and compare it to $v_2(n^2)$.

The sum over $d$ can be partitioned into pairs $\{d,2d\}$ where $d$ ranges over all divisors of $n$ such that $n/2d$ is odd and square-free. For each pair, the contribution to the sum is 
\[
(-1)^n\mu(n/2d)\left(\binom{4d-1}{2d-1}-(-1)^d\binom{2d-1}{d-1}\right).
\]
Lemma \ref{lev2Binomial} implies that the $2$-divisibility of this expression is given by 
$$2+2v_2(d)+s_2(d-1)=v_2(n^2)+s_2(d-1),$$ 
by taking into account that $v_2(d)=v_2(n)-1$. Since $s_2(d-1)=0$ if and only if $d=1$, we conclude that $Q_n(1)$ is odd if and only if $d=1$ contributes to the sum, which is equivalent to the square-free property of $n$.
\end{enumerate}

\end{proof}

\begin{proof}[Proof of Theorem \ref{thDivLogFCase2}]
A specialization of \ref{eqHAsASum} with $p=2$ gives
\begin{align*}
(q-1)\left(h(q,z)-\frac12 h(q^2,z^2)\right)&=(q-1)\sum_{l=1}^{\infty}\frac{k(q^{2l-1},z^{2l-1})}{2l-1}\\
&=\sum_{l=1}^{\infty}\sum_{m=1}^{\infty}\frac{q-1}{q^{2l-1}-1}\frac{(-1)^{m-1}Q_m(q^{2l-1})}{(2l-1)q^{(2l-1)\binom{d}{2}}}z^{(2l-1)m},
\end{align*}
and therefore
\begin{equation}\label{eqH2CoeffSum}
(q-1)[z^n]\left(h(q,z)-\frac12 h(q^2,z^2)\right)=\sum_{(2l-1)m=n}\frac{q-1}{q^{2l-1}-1}\frac{(-1)^{m-1}Q_m(q^{2l-1})}{(2l-1)q^{(2l-1)\binom{d}{2}}}.
\end{equation}
We point out that the summand $$\frac{q-1}{q^{2l-1}-1}\frac{(-1)^{m-1}Q_m(q^{2l-1})}{(2l-1)q^{(2l-1)\binom{d}{2}}}$$ belongs to $\Z_2$, so we can refer to its parity. Indeed, the parity of this summand is equal to the parity of $Q_m(1)$, so using Lemma~\ref{leQn1Parity} we know it is odd if and only if $m$ is square-free.

Therefore, the number of odd summands in the sum \ref{eqH2CoeffSum} is equal to the number of factorizations of $n$ into an odd number $2l-1$ and a square-free number $m$.

Suppose that $n\geq3$. We prove that the number of such factorizations of $n$ is always even. Indeed, if $4\mid n$ then no such factorization is possible. Otherwise, there are $2^{\omega_2(n)}$ factorizations, where $\omega_2(n)$ is the number of distinct odd prime factors of $n$, and since $n\geq 3$ we always have $\omega_2(n)>0$.

Thus we have proved $(q-1)[z^n]\left(h(q,z)-\frac12 h(q^2,z^2)\right)$ is even whenever $n\geq3$, thereby establishing the theorem.
\end{proof}

\setcounter{equation}{0}
\setcounter{theorem}{0}
\global\def\thesection{\mbox{D}}

\section*{Appendix D: Dwork's Lemma and the proof of Theorem \ref{thDivLogFCaseD}}
The Dieudonn\'{e}--Dwork Lemma (see, for example, Chapter 14 of \cite{MR566952}), is a fundamental result about the relationship between the $p$-adic properties of a power series and its exponential. It says the following.
\begin{lemma}[{\cite[Chapter 14]{MR566952}}]\label{leDwork}
Let $f\in \Q_p[[z]]$ be a power series. Then $\exp f\in \Z_p[[z]]$ if and only if
\[
p f(z)-f(z^p)\in \Z_p[[z]].
\]
\end{lemma}

From this lemma we conclude the following fact.
\begin{lemma}\label{leZpModule}
The set
\[
\{f\in \Q_p[[z]]\mid \exp f\in \Z_p[[z]]\}
\]
is a $\Z_p$-module.
\end{lemma}
\begin{proof}
Lemma \ref{leDwork} implies that the set $\{f\in \Q_p[[z]]\mid \exp f\in \Z_p[[z]]\}$ is the preimage of the $\Z_p$-module $\Z_p[[z]]$ under the $\Z_p$-linear map $f\mapsto pf(z)-f(z^p)$, and therefore is also a $\Z_p$-module.
\end{proof}

\begin{lemma}\label{leDivOfFCaseInfty}
Let $F$ be defined as in \eqref{eqGenFunDef}. Then, for any prime $p$ and any prime power $q$, we have $\lim_{k\to\infty}F(C_{p^k},q,z)\in \Zp[[z]]$.
\end{lemma}
\begin{proof}
The quantity $|GL_n(\Fq)|[z^n]F(C_{p^k},q,z)$ is the number of homomorphisms from $C_{p^k}$ to $GL_n(\Fq)$, and is equivalently the number of elements with order dividing $p^k$ in the group. Letting $k\to\infty$, we obtain the number of elements whose order is equal to a power of $p$. According to the Frobenius Theorem, the $p$-divisibility of this number is at least the $p$-divisibility of $|GL_n(\Fq)|$, which establishes the claim.
\end{proof}

Using Lemmas \ref{leZpModule} and \ref{leDivOfFCaseInfty}, we are able to prove Theorem \ref{thDivLogFCaseD}.
\begin{proof}[Proof of Theorem \ref{thDivLogFCaseD}]
Theorem \ref{thFormulaOfLogF} implies the identity
\begin{equation}\label{eqFIdentityCaseD}
\log F(G,q;z)-\frac1d\log F(G,q^d;z^d)=h(q,z)-\frac1d h(q^d,z^d).
\end{equation}
holds for every Abelian $p$-group $G$. We take $G=C_{p^k}$ and let $k\to+\infty$ in this identity. Lemmas~\ref{leZpModule} and \ref{leDivOfFCaseInfty} imply that the left-hand side of \ref{eqFIdentityCaseD} is a linear combination of elements in the $\Z_p$-module $\{f\in \Q_p[[z]]\mid \exp f\in \Z_p[[z]]\}$, with coefficients in $\Z_p$. Therefore, the right-hand side of \ref{eqFIdentityCaseD} also belongs to this module, which concludes the proof.
\end{proof}

\bibliographystyle{siam}
\bibliography{hom_q_div}
\end{document}